\documentclass[english]{lematema}

\makeatletter 
\def\@maketitle{%
  \newpage
  \ifx\@MSC\empty\else\thanks{%
  \textit{AMS 2010 Subject Classification:} \@MSC
  }\fi%
  \ifx\@keywords\empty\else\thanks{%
  \textit{Keywords:} \@keywords
  }\fi%
  \begin{center}%
  \let \footnote \thanks
    {\def\\{\protect\linebreak}%
    \fontsize{12}{15}\selectfont\bfseries \MakeUppercase{\@title}
    \ifx\sub@title\empty\else\vskip5pt\fontsize{10}{14}\selectfont\bfseries \MakeUppercase{\sub@title}\fi
    \the\@titlenote
    \par}%
    \if@author
    \vskip 0.5truecm%
    {\fontsize{11.5}{14}\selectfont
      \begin{tabular}[t]{c}%
        \let\authorline@sep\tabularnewline
        \ifshortauthors\@shortauthor\else\@author\fi
      \end{tabular}\par
    }%
    \fi
\end{center}%
\par
\vskip 0.5truecm}

\makeatother 

\title{Computing a minimal resolution over the Steenrod algebra}
\titlemark{Minimal resolution}
\author{Christian Nassau}

\address{%
Jenaer Weg 31\\
65931 Frankfurt\\
Germany\\
\email{nassau@nullhomotopie.de}%
}

\date{2019/10/04}


\usepackage{csquotes}
\usepackage{mathtools}
\usepackage{algorithm2e}
\usepackage{tikz}
\usetikzlibrary{cd}
\usepackage{pgfplots}
\pgfplotsset{width=8cm,compat=1.14}

\SetKwComment{LineComment}{}{}
\SetKwComment{Comment}{(}{)}
\SetCommentSty{emph}

\DeclareMathOperator{\Ext}{Ext}
\DeclareMathOperator{\Tor}{Tor}
\DeclareMathOperator{\im}{im}
\DeclareMathOperator{\Sq}{Sq}
\DeclareMathOperator{\sig}{sig}
\DeclareMathOperator{\xmod}{\,mod\,}
\newcommand{\Sig}{{\mathcal S}}
\newcommand{\PP}{{\mathcal P}}

\newcommand{\lefthopfquot}{\mathop\backslash\!\!\!\mathop\backslash}
\newcommand{\FF}{{\mathbb F}}
\newcommand{\NN}{{\mathbb N}}
\newcommand{\us}{_\ast}

\newcommand{\cnsetof}[2]{\left\{#1\,\mid\,#2\right\}}

\begin{document}

\maketitle

\begin{abstract}
We describe an algorithm that allows to compute a minimal 
resolution of the Steenrod algebra. 
The algorithm has built-in knowledge 
about vanishing lines for the cohomology of sub Hopf algebras
of the Steenrod algebra which makes it both faster and more economical
than the generic approach.
\end{abstract}

\section{Introduction}
Let $A$ denote the Steenrod algebra at a prime $p$ and let $k=\FF_p$.
The cohomology of $A$ is, by definition, the $\Ext$ group $\Ext_A(k,k)$.
It features prominently in algebraic topology as the $E_2$ term of the
Adams spectral sequence for the computation of the stable homotopy groups
of the sphere (see \cite{MR860042}).  

Machine computations of the cohomology of $A$ have a long history
and there is considerable
current activity in the field.
While there are other legitimate approaches
(e.g.~the May spectral sequence \cite[Ch.~3.2]{MR860042}
or the Lambda algebra \cite{MR818916})
the most promising
route for purely mechanical computations 
seems to be the computation of the minimal resolution
of the ground field $k$ as pioneered by Bruner \cite{MR1224908}.
The main obstacle here is the enormous size of the resolution:
a computation for $p=2$ up to topological dimension $200$, 
for example, will require the computation
of kernels and cokernels of matrices over $k$ with hundreds of thousands
of rows and columns
(see \cite[Abb.~2.14]{zbMATH02190992} for a chart showing the growth rate of the resolution). 
Carrying out such computations in a reasonable time
seems well beyond the capabilities
of current computing technology.

The author's contribution to this story is the discovery
of a powerful shortcut based on vanishing lines for the cohomology
of subalgebras of $A$.
The author lectured about these results in
Oberwolfach in 1997
and these shortcuts became the basis
of his PhD dissertation \cite{zbMATH02190992}.
For reasons long lost in time an English language account
of these results has never been published.
This note is meant to remedy that ommission.

We have chosen not to give detailed proofs 
of the main theorems 
since this would require the introduction of cumbersome notation 
that would obscure the simple idea behind our approach. 
The mathematics involved is completely elementary
and a reader who works out the examples that we give
in Lemmas \ref{lem:sq1a}, \ref{lem:sq1b} and \ref{lem:exa1}
will have no problems filling in the details in the more general cases.

\section{The algorithm}\label{sec:algbasic}

We will assume $p=2$ throughout to simplify the exposition.
All results generalize to odd primes in a straightforward way
(for details see \cite{zbMATH02190992} or \cite{steenrodlib}).

We let $C\us$ denote the minimal resolution that we wish to compute. 
Every $C_s$ is a free $A$-module with a chosen set of generators $G_s\subset C_s$.
The differential $d:C_s\rightarrow C_{s-1}$ is described by keeping a list of the
$d(g_k)$ for $g_k \in G_s$.

One works by double induction on the internal degree $t$ 
of the Steenrod algebra and the homological degree $s$ of the resolution.
At each step $C\us$ is a partially complete resolution
below some bidegree $(s,t)$:
one has $H(C\us)_{p,q} = 0$ if $p<s$ or $q<t$,
but in $(s,t)$ itself the homology $H(C\us)_{s,t}$ might be non-zero.
If it is non-zero we introduce new generators in $C_{s+1}$ that kill the offending
homology classes. 
We can describe this procedure more formally as follows:

\begin{center}
    \begin{algorithm}[H]\label{alg:naive}
        \SetAlgoLined
        \DontPrintSemicolon
        \KwIn{A partial resolution below bidegree $(s,t)$}
        \KwResult{An extension of the resolution to $(s,t)$}
        $M \longleftarrow$ compute matrix of $d:C_{s,t}\rightarrow C_{s-1,t}$\;
        $K \longleftarrow$ basis of kernel of $M$\;
        $N \longleftarrow$ compute matrix of $d:C_{s+1,t}\rightarrow C_{s,t}$\;
        $Q \longleftarrow$ basis of quotient $K/\im N$, i.e.~of $H(C\us)_{s,t}$\;
        \For{$q\in Q$}{
            $x \longleftarrow$ pick a representative of $q$\;
            Introduce new generator $g\in C_{s+1,t}$ with $dg = x$.
        }
        \caption{The naive algorithm}
    \end{algorithm}
\end{center}

This is the basic algorithm as it applies to any connected, graded algebra.
To see how this can be improved given more specific knowledge about the algebra, 
consider the short exact sequence
\begin{equation}\label{sqex1}
    \begin{tikzcd}
        \Sq^1 C\us \arrow[r, "\text{incl.}", hook] & C\us \arrow[r, "\Sq^1\cdot", two heads] & \Sq^1 C\us
    \end{tikzcd}
\end{equation}
\begin{lemma}\label{lem:sq1a}
    If $C\us$ is a partially complete resolution below $(s,t)$ and if $t-s>1$
    the map $H\left(C\us\right)_{s,t} \rightarrow
    H\left(\Sq^1 C\us\right)_{s,t+1}$ is an isomorphism.
\end{lemma}
\begin{proof}
The associated long exact sequence in homology 
contains
\begin{equation}\label{les1}
    \begin{tikzcd}[column sep=small]
        H\left(\Sq^1 C\us\right)_{s,t} \arrow[r] & 
        H\left(C\us\right)_{s,t} \arrow[r] & 
        H\left(\Sq^1 C\us\right)_{s,t+1} \arrow[r] & 
        H\left(\Sq^1 C\us\right)_{s-1,t} 
    \end{tikzcd}
\end{equation}
so it suffices to show that the left and right hand groups are zero.

Let $A(0)$ denote the exterior algebra $\FF_2\{1,\Sq^1\}$.
If $C\us$ was already a complete resolution 
one would have
$H\left(\Sq^1C\us\right)_{p,q} = \Ext^{p,q-1}_{A(0)}(k)$
since $\Sq^1C\us \cong k\otimes_{A(0)} C\us$
(up to a degree shift of $1$) 
and $C\us$ would function as an $A(0)$-resolution of $k$.
For a partial resolution that identification holds true through a range
and one can check that it applies to the boundary terms in (\ref{les1}).
Since $\Ext^{p,q}_{A(0)}=0$ for $q-p>0$ that proves the Lemma.
\end{proof}
The Lemma shows that for $t-s>1$ the computation of $Q$ in Algorithm \ref{alg:naive}
can be carried out in $\Sq^1C\us$ which is approximately only half as big as $C\us$.

This alone does not quite suffice for the completion of the inductive step, though: 
the algorithm needs a representative cycle $x$ from $C_{s,t}$, but
a computation of $H(\Sq^1C\us)_{s,t+1}$ will only produce a cycle $\Sq^1 x'$
in $\Sq^1 C\us$. 
Writing $x=x'+\Sq^1 x''$ we thus still need to determine the unknown component $x''$.
\begin{lemma}\label{lem:sq1b}
    If $t-s>0$ we can determine $\Sq^1 x''$ 
    by solving $d\left(\Sq^1x''\right)=-dx'$ in $\Sq^1C\us$.
\end{lemma}
\begin{proof}
    Firstly, one has $\Sq^1 dx' = d\Sq^1x' = 0$, so $dx'$ lies in $\Sq^1C_{s-1,t-1} = \left(\Sq^1C\us\right)_{s-1,t}$
    by the exactness of (\ref{sqex1}).
    Arguing as in Lemma \ref{lem:sq1a} we find that $H\left(\Sq^1C\us\right)_{s-1,t}$
    computes $\Ext^{s-1,t-1}_{A(0)}(k)$. The assumption $t-s>0$ guarantees that this
    vanishes, so there is indeed a $\Sq^1x''\in \Sq^1C_{s-1,t-1}$ with boundary $-dx'$.
\end{proof}
Together Lemmas \ref{lem:sq1a} and \ref{lem:sq1b} show that for $t-s>1$ one can 
trade the single homology calculation in Algorithm \ref{alg:naive}
against one homology calculation in $\Sq^1 C\us$ and the solution of one lifting problem
in $\Sq^1 C\us$. 
In large dimensions this is a considerable improvement:
the matrix that represents the differential in $\Sq^1C\us$ will need roughly
just a quarter of the space that would be required to store the full 
differential; even though two such matrices are needed, they are needed sequentially,
so the same space can be reused; and the linear algebra routines
for the computation of kernel and quotient will run a lot faster since their running times
are typically more than quadratic in the size of the input matrices.

The real power of this trick, however, is that it can be iterated. Consider the
short exact sequences
\begin{equation}
\begin{tikzcd}[row sep=tiny]
    \Sq(1,1) C\us \arrow[r, "\text{incl.}", hook] & 
    \Sq^1 C\us \arrow[r, "{\cdot\Sq(0,1)}", two heads] & \Sq(1,1) C\us 
    \\
    \Sq(3,1) C\us \arrow[r, "\text{incl.}", hook] & 
    \Sq(1,1) C\us \arrow[r, "\cdot\Sq^2", two heads] & \Sq(3,1) C\us
\end{tikzcd}
\end{equation}
The homology of $\Sq(1,1)C\us$ and $\Sq(3,1)C\us$ is approaching $\Ext_{B}^{p,q}(k)$
where $B$ is, respectively, the exterior algebra $E$ on $\Sq^1$ and $\Sq(0,1)$
or the subalgebra $A(1)$
(as usual we let $A(n)\subset A$ denote the sub Hopf algebra spanned by 
$\Sq^1,\ldots,\Sq^{2^n}$).
These both vanish if $q>3p$ and there is the following
straightforward generalization of
Lemma \ref{lem:sq1a} and \ref{lem:sq1b}.
\begin{lemma}\label{lem:exa1}
    Let $C\us$ be a partially complete resolution below $(s,t)$ and assume $t>3(s+1)$.
    Then the map $C\us \rightarrow \Sq(3,1)C\us$ with $x\mapsto \Sq(3,1)x$
    induces an isomorphism
    $H\left(C\us\right)_{s,t} \cong H\left(\Sq(3,1)C\us\right)_{s,t+6}$.
    Furthermore, any cycle $\Sq(3,1)x_0\in \Sq(3,1)C_{s,t}$
    can be completed to a cycle $x=x_0+x_1+\cdots + x_7$ in $C_{s,t}$ by solving 
    $7$ subsequent lifting problems in $\Sq(3,1)C\us$.
\end{lemma}
\begin{proof}
To establish the claimed isomorphism one needs to look at the long exact sequences
\newcommand{\subscr}[1]{_{\!\mathmakebox[.1em][l]{#1}}}
\newcommand{\subscrx}[1]{_{\!{#1}}}
\begin{small}
\begin{equation*}
\begin{tikzcd}[column sep=small, cramped, row sep=tiny]
    H\Big(\Sq^{1} C\us\Big)\subscr{s,t} \arrow[r] & 
    H\Big(C\us\Big)\subscr{s,t} \arrow[r] & 
    H\Big(\Sq^{1} C\us\Big)\subscr{s,t+1} \arrow[r] & 
    H\Big(\Sq^{1} C\us\Big)\subscrx{s-1,t} 
    \\
    H\Big(\Sq{(1,1)} C\us\Big)\subscr{s,t+1} \arrow[r] & 
    H\Big(\Sq^{1} C\us\Big)\subscr{s,t+1} \arrow[r] & 
    H\Big(\Sq{(1,1)} C\us\Big)\subscr{s,t+4} \arrow[r] & 
    H\Big(\Sq{(1,1)} C\us\Big)\subscrx{s-1,t+1} 
    \\
    H\Big(\Sq{(3,1)} C\us\Big)\subscr{s,t+4} \arrow[r] & 
    H\Big(\Sq{(1,1)} C\us\Big)\subscr{s,t+4} \arrow[r] & 
    H\Big(\Sq{(3,1)} C\us\Big)\subscr{s,t+6} \arrow[r] & 
    H\Big(\Sq{(3,1)} C\us\Big)\subscrx{s-1,t+4} 
\end{tikzcd}
\end{equation*}
\end{small}%
The terms at the end compute, respectively, $\Ext_{A(0)}^{p,t-1}$, $\Ext_E^{p,t-3}$ and $\Ext_{A(1)}^{p,t-2}$ for $p=s,s-1$.
Using $\Ext_B^{p,q}=0$ for $q>3p$ one finds that they all vanish if $t>3(s+1)$.

The recovery of a cycle $x\in C_{s,t}$ from knowledge of the cycle $\Sq(3,1)x\in \Sq(3,1)C_{s,t}$ is a straightforward
diagram chase that we leave to the reader. It requires the exactness of $\Sq(3,1)C_{s-1,t+6-p}$ for $p=1,\ldots,6$.
These groups relate to $\Ext_{A(1)}^{s-1,t-p}$ which are again zero for $t>3(s+1)$.
\end{proof}
The important points to remember are
\begin{enumerate}
    \item
A vanishing result for various $\Ext_B^{s,\ast}$ and $\Ext_B^{s-1,\ast}$ 
is used to reduce the homology calculation in $C_{s,t}$ 
to a homology calculation in a space of much smaller dimensions.
    \item 
Similar vanishing results for various $\Ext_B^{s-1,\ast}$ are required to use the same reduction to recover the full cycle. 
\item 
The lifting problems for the recovery of the full cycle
are enumerated by the Milnor basis elements $\Sq(R)$ in $B$ and they take
place in degree $t-\vert\Sq(R)\vert$.
\end{enumerate}

To give a more formal account of the algorithm we start with a sub
Hopf algebra\footnote{%
There is an interesting limiting case where $B$ 
is just a sub algebra: this is discussed in Lemma \ref{thm:above2} below.} 
$B\subset A$.
Recall from 
\cite[Ch.~15,~Thm.~6]{MR738973}
that such a $B$ 
is described by a profile function
$p:\NN\rightarrow \NN\cup\{\infty\}$. 
A vector space basis of $B$
is given by those Milnor basis elements $\Sq(R)$ such that $0\le r_j < 2^{p(j)}$.

We will always assume $B$ to be finite, since a computation in a finite dimension
does not see the difference between $B$ and its truncation to some $A(N)$ with $N\gg0$.
Let $\Sig_B = \cnsetof{R}{\Sq(R)\in B}$ and call it the set of \enquote{signatures} in $B$.
We say that $R$ and $S$ have the same $B$-signature if $S_j\equiv R_j\xmod 2^{p(j)}$
for all $j$.
We denote this by $\Sq(R) \simeq_B \Sq(S)$.
For every $S$ there is a unique $R\in\Sig_B$ such that $\Sq(S)\simeq_B\Sq(R)$;
this is called the $B$-signature of $\Sq(S)$ and written as $\sig_B(\Sq(S))$.

Having the same signature defines the \enquote{signature decomposition} 
$$A=\sum\nolimits^\oplus_{R\in\Sig_B} E_RA,\qquad E_RA = \FF_2\cnsetof{\Sq(S)}{\sig_B(S) = R}.$$
We will shortly put an ordering on the signatures $\Sig_B = \{R_0<R_1<\cdots<R_k\}$. 
This allows us to consider the \enquote{signature filtration}
$$F_{R_k}A \subset \cdots \subset F_{R_1}A \subset F_{R_0}A = A$$
with $F_RA = \sum_{S\ge R} E_S(A)$.
We want every $F_RA$ to be a right $A$-submodule of $A$ because we can then 
extend the filtration to our resolution via $F_R C\us = F_RA \otimes_A C\us$.
\begin{lemma}\label{lem:badm}
Let $B\subset A$ be the sub Hopf algebra associated to a profile function
$p$ with $p(i+j)\ge p(i)-j$ for every $i,j\ge 1$. 
Let $\PP$ be the set of the $P_t^s\in B$ and choose an ordering 
$\PP = \{P_{t_1}^{s_1} > \cdots > P_{t_n}^{s_n}\}$
with $P_t^{s}<P_{t'}^{s'}$ whenever $t<t'$.
For $R\in\Sig_B$ 
consider the binary decomposition $r_j = \sum_{t_k=j} 2^{s_k} \varepsilon_k$ of each $r_j$.
Order the signatures via the lexicographic ordering of the bit vector $\tilde R = (\varepsilon_1,\ldots,\varepsilon_n)$.
Then every $F_RA$ is stable under right multiplication by $A$.
Furthermore $E_RA$ (which is a right $A$-module as a quotient of $F_RA$) is up to degree shift by $\vert R\vert$
isomorphic to $B\lefthopfquot A$.
\end{lemma}
For the proof of the Lemma we should recall Milnor's multiplication algorithm
(see \cite{MR0099653} or \cite[Ch.~15]{MR738973}). This expresses a multiplication
$$\Sq(R)\cdot \Sq(S) = \sum_X \beta_{R,S,X}\cdot Sq(T)$$
as a sum over certain matrices $X=(x_{i,j})$ such that
\begin{enumerate}
    \item the weighted row sums decompose the first factor: $r_i = \sum_j 2^jx_{i,j}$
    \item the column sums decompose the second factor: $s_j = \sum_i x_{i,j}$
    \item the diagonal sums decompose the result: $t_k = \sum_{i+j=k} x_{i,j}$
\end{enumerate}
The coefficient $\beta_{R,S,X}$ is nonzero if and only if the diagonal decomposition of the $t_k$
is bitwise disjoint.
\begin{lemma}\label{lem:milmult}
With the assumptions of Lemma \ref{lem:badm}, call a matrix $X=(x_{i,j})$ $B$-trivial
if $x_{i,j} \equiv 0 \xmod 2^{p(i)-j}$ holds whenever $j\le p(i)$.
Then 
$$\Sq(R)\cdot\Sq(S) = \sum_{B\text{-trivial}\, X} \beta_{R,S,X}\cdot Sq(T) + \text{terms with signature $>R$.}$$
\end{lemma}
\begin{proof}
By design, a $B$-trivial $X$ will have $\sig_B(R) \le \sig_B(T)$:
from $0\equiv x_{i,j}$ $\xmod$ $2^{p(i)-j}$ and $p(i)-j\ge p(i+j)$ one finds
$t_k\equiv r_k + x_{0,k}\xmod 2^{p(k)}$ and this sum must be disjoint.
Hence the multiplication with $B$-trivial $X$ can only add non-zero bits to $R$ which cannot lower the 
signature.

If $X$ is not $B$-trivial one needs to chase the possible movements of a bit $2^k\in r_j$.
This can only be removed from $r_j$ by moving to the right in $x_{j,\ast}$, hence affecting a
bit in $t_k$ with $k>j$. Since we required $P_k^\ast>P_j^\ast$ that bit is more significant
than the original $2^k\in r_j$ and the signature is increased.\end{proof}

The proof of Lemma \ref{lem:badm} is now immediate. 
Note that to compute a right multiplication 
$$E_RA \xrightarrow{x\mapsto x \cdot \Sq(S)} E_RA$$
it suffices to enumerate the $B$-trivial matrices $X$;
the multiplication is therefore insensitive to the $B$-signature of the first factor.
Hence all $E_RA$ are isomorphic to $\tau_B\cdot A \cong B\lefthopfquot A$
where $\tau_B = \Sq(R\textsubscript{max})$ is the largest dimensional element of $B$.

We will from now on only consider sub Hopf algebras $B$ as in Lemma \ref{lem:badm}
with their compatible signature ordering; such $B$ will be called admissible.

We now have a signature filtration $F_RC\us$ on the partial resolution.
Note that $H(E_RC\us)_{p,q} \cong H(B\lefthopfquot A \otimes_A C\us)_{p,q}$ 
approximates 
$\Tor^A_{p,q-\vert\Sq(R)\vert}(B\lefthopfquot A)$ 
which is dual to  $\Ext^{p,q-\vert R\vert}_B(k)$.
Hence a vanishing result for $\Ext_B(k)$ will translate
to a corresponding exactness assertion for $E_RC\us$.
With these preparations the proposed new algorithm can then be formalized as in Algorithm \ref{alg:new} (see page \pageref{alg:new}).

As explained earlier, the algorithm is not automatically applicable everywhere:
it only works and produces valid results when the bidegrees
$(s,t-\vert R\vert)$, $(s-1,t-\vert R\vert)$ for various $R\in\Sig_B$ are contained in a 
known vanishing region for the cohomology of the subalgebra $B$.
There are easily determined vanishing regions for the $E_2$-term of the May 
spectral sequence for $\Ext_B$ that we can use.
To state them let $Q_s=P_{s+1}^0$ denote the usual Bockstein operation. 
\begin{lemma}\label{bvanishing}
Let $B\subset A$ be a finite sub Hopf algebra and let $s\textsubscript{min}$, $s\textsubscript{max}$
be the smallest, resp.~largest $s$ with $Q_s\in B$. Then $\Ext_B^{s,t}(k)=0$ if either
$t<s\cdot\vert Q_{s\textsubscript{min}}\vert$ 
or 
$t>s\cdot\vert Q_{s\textsubscript{max}}\vert$.
\end{lemma}
One can thus choose between working above the vanishing line based on $s\textsubscript{min}$
or below the vanishing line for $s\textsubscript{max}$.
We discuss the merits of these choices in the next sections.

\begin{center}
    \begin{algorithm}\label{alg:new}
        \SetAlgoLined
        \DontPrintSemicolon
        \KwIn{A partial resolution below bidegree $(s,t)$}
        \KwIn{An admissible subalgebra $B\subset A$ with its ordering of $\Sig_B$}
        \KwResult{An extension of the resolution to $(s,t)$}
        \LineComment*[h]{Start by computing the homology of $E_0C\us$}\;
        $M \longleftarrow$ matrix of $d:E_0C_{s,t}\rightarrow E_0C_{s-1,t}$\;
        $K \longleftarrow$ basis of kernel of $M$\;
        $N \longleftarrow$ matrix of $d:E_0C_{s+1,t}\rightarrow E_0C_{s,t}$\;
        $Q \longleftarrow$ basis of quotient $K/\im N$, i.e.~of $H(E_0C\us)_{s,t}$\;
        \If{$Q$ not empty}{
            \LineComment*[h]{Set up approximate boundaries for the new generators}\;
            \For{$q_i\in Q$}{
                $x_i \longleftarrow$ a representative of $q_i$ in $E_0C_{s,t}$
                \Comment*[r]{will become $d(g_i)$}
                $d_i \longleftarrow d(x_i)$  
                \Comment*[r]{represents $d^2(g_i)$, should be zero at the end}
            }
            \LineComment*[h]{Extract error terms $e_i$ from $d^2(g_i)$ 
            and compute corrections to $d(g_i)$}\;
            \For{$R\in\Sig_B,\,R\not=0$ (process these in order)}{
                $M \longleftarrow$ matrix of $d:E_RC_{s,t}\rightarrow E_RC_{s-1,t}$\;
                \For{$q_i\in Q$}{
                    $e_i \longleftarrow$ extract the summands of $d_i$ from $E_RC_{s-1,t}$\;
                    $f_i \longleftarrow$ solution of $M\cdot f_i = e_i$\;
                    $x_i \longleftarrow x_i - f_i$\;
                    $d_i \longleftarrow d_i - d(f_i)$\;
                }
            }
            \LineComment*[h]{Introduce new generators}\;
            \For{$q_i\in Q$}{
                Verify that $d_i = 0$\;
                Introduce new generator $g_i\in C_{s+1,t}$ with $dg_i = x_i$.
            }
        }
        \caption{The new algorithm}
    \end{algorithm}
\end{center}

\section{Working below the vanishing line}
Suppose $B\subset A(n)$ and let $\tau_B$ be the maxmimum 
of the dimensions of the $\Sq(R)$ in $B$. 
With $\rho_n=\vert Q_{n+1}\vert = 2^{n+1}-1$ one has
\begin{thm}\label{thm:below}
    Suppose $B\subset A(n)$ and $t>\rho_n\cdot (s+1) + \tau_B$.
    Then algorithm \ref{alg:new} is applicable.
\end{thm}
\begin{proof}
This is a straightforward diagram chase along the lines of
Lemma \ref{lem:sq1a} and \ref{lem:sq1b}.
Details can be found in \cite{zbMATH02190992}, Satz~2.2.15.
\end{proof}
For a fixed bidegree $(s,t)$ this criterion
yields a finite number of choices for applicable subalgebras $B$.
Choosing the best one then amounts to finding that $B$ for which the
dimensions of the $E_RC\us$ over all $\Sig_B$ becomes smallest.

In practice we have just evaluated the size of the first piece $E_0C\us$
to make the choice of $B$; experience shows that the initial homology
calculation using $E_0C\us$ is 
facing larger matrices than the
subsequent lifting problems.
In our actual implementation we have only implemented a  simplified 
search for $B$ using an ordering of the $P_t^s\in A$, first via $t+s$,
then by $s$. We then considered only those $B$ that are spanned by an initial
segment of the $P_t^s$, i.e.~ $A(0)$, $E(\Sq^1,\Sq(0,1))$, $A(1)$,
$A(1)\cdot E(\Sq(0,0,1))$, etc. In practice this seems to give a sufficiently
good choice.

We illustrate the working of the algorithm with some statistics
from the computation for $p=2$ in topological dimension $120$.
Algorithm \ref{alg:new} is applicable there with $B=A(2)$.
The dimensions of the vector spaces for $s=10,\ldots,14$ are
given in the following table.
\begin{center}
\begin{tabular}{c|ccc|ccc}
  $s$ &  
  $C_{s-1}$ & $C_s$ & $C_{s+1}$ &  
  $E_0C_{s+1}$ & $E_0C_s$ & $E_0C_{s-1}$ \\
  \hline
  10 & 29087 & 23997 & 19609 & 482 & 586 & 706 \\
  11 & 25029 & 20477 & 18385 & 447 & 494 & 604 \\
  12 & 21388 & 19213 & 18325 & 455 & 474 & 537 \\
  13 & 20070 & 19156 & 16601 & 396 & 457 & 489 \\
  14 & 19993 & 17350 & 14437 & 364 & 436 & 494 
\end{tabular}
\end{center}
The dimensions of the matrices that are encountered in the
subsequent lifting problems can be seen in
Figure \ref{fig:decompb}.

\begin{figure}
    \begin{center}
\begin{tikzpicture}
    \begin{axis}[%
        xlabel={number of rows}, ylabel={number of columns},%
        scatter/classes={%
        hom10.0.260={mark=square*,  scale=.5,blue},%
        moh10.0.260={mark=square*,  scale=.5,blue},%
        lft10.0.260={mark=triangle*,scale=.5,blue},%
        hom11.0.262={mark=square*,  scale=.5,red},%
        moh11.0.262={mark=square*,  scale=.5,red},%
        lft11.0.262={mark=triangle*,scale=.5,red},%
        hom12.0.264={mark=square*,  scale=.5,olive},%
        moh12.0.264={mark=square*,  scale=.5,olive},%
        lft12.0.264={mark=triangle*,scale=.5,olive},%
        hom13.0.266={mark=square*,  scale=.5,purple},%
        moh13.0.266={mark=square*,  scale=.5,purple},%
        lft13.0.266={mark=triangle*,scale=.5,purple},%
        hom14.0.268={mark=square*,  scale=.5,magenta},%
        moh14.0.268={mark=square*,  scale=.5,magenta},%
        lft14.0.268={mark=triangle*,scale=.5,magenta},%
            c={mark=o,draw=black}}]
            \addplot[scatter,only marks, scatter src=explicit symbolic] table[meta=label] {log.10.0.120};
            \addplot[scatter,only marks, scatter src=explicit symbolic] table[meta=label] {log.11.0.120};
            \addplot[scatter,only marks, scatter src=explicit symbolic] table[meta=label] {log.12.0.120};
            \addplot[scatter,only marks, scatter src=explicit symbolic] table[meta=label] {log.13.0.120};
            \addplot[scatter,only marks, scatter src=explicit symbolic] table[meta=label] {log.14.0.120};
        \end{axis}
\end{tikzpicture}
\end{center}
\caption{Application of Algorithm \ref{alg:new} to the computation of 
$C_s$ for $s=10,\ldots,14$
and $t-s=120$ with $B=A(2)$ (the corresponding colors are blue, red, olive, purple, magenta). 
The squares represent 
the matrices of the initial homology computation; the triangles correspond to the lifting problems.}
\label{fig:decompb}
\end{figure}
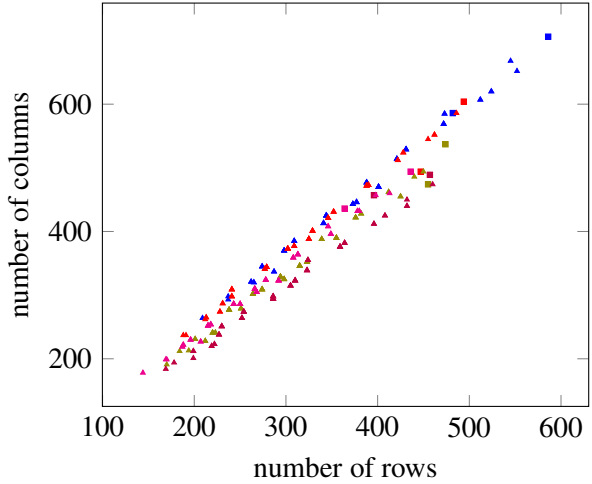

\section{Working above the vanishing line}
The subalgebras that were chosen in the last section 
become bigger as $(x,y)=(t-s,s)$ approaches the $x$-axis; 
conversely, when $s$ gets bigger they tend to become smaller,
hence less useful. For these cases a different choice of $B$
presents itself:
let $F(n)=\FF_2\cnsetof{\Sq(R)}{r_1=\cdots=r_n=0}$.
The smallest Bockstein in $F(n)$ is $Q_{n+1}$, hence Lemma \ref{bvanishing}
gives
\begin{thm}\label{thm:above}
Suppose $B\subset F(n)$ and $t<(2^{n+1}-1)s$.
Then algorithm \ref{alg:new} is applicable.
\end{thm}
\begin{proof}
    This is again a straightforward diagram chase along the lines of
    Lemma \ref{lem:sq1a} and \ref{lem:sq1b}.
    For details see \cite{zbMATH02190992}, Satz~2.2.19.
\end{proof}

To see that this gives a powerful choice of $B$ consider the problem of 
computing the zeroes above the Adams vanishing region, i.e.~confirming mechanically
Adams' theorem that $\Ext_A^{s,t}=0$ if $0<t-s<2s-3$ (see  \cite{MR194486}).
The naive algorithm would have to verify the exactness of 
$$\begin{tikzcd}[column sep=1cm]
    Ah_0^{s-1} 
    & Ah_0^s \arrow[l, "{\cdot\Sq(1)}"']
    & Ah_0^{s+1} \arrow[l, "{\cdot\Sq(1)}"']
\end{tikzcd}$$
in degree $t-s$; if $t-s=200$ these vector spaces have dimension $\approx 15000$
which makes the computation of kernels and images challenging.
Theorem \ref{thm:above} allows to pick $B=F(1)$ in this region, so 
the homology calculation in Algorithm \ref{alg:new} will 
instead look at the sequence
\begin{equation}\label{eq:fqex}
\begin{tikzcd}[column sep=1cm]
    F(1)\lefthopfquot Ah_0^{s-1} 
    & F(1)\lefthopfquot Ah_0^s \arrow[l, "{\cdot\Sq(1)}"']
    & F(1)\lefthopfquot Ah_0^{s+1} \arrow[l, "{\cdot\Sq(1)}"']
\end{tikzcd}\end{equation}
Here $F(1)\lefthopfquot A\cong\FF_2\cnsetof{\Sq^k}{k\ge 0}$ is just one-dimensional
in every degree so the computation is pretty trivial
(and in fact independent of the topological dimension $t-s$).

We illustrate the effect of using $B=F(2)$ in the following table
which lists the dimensions of the vector spaces for $s=35$ and $t-s=114,\ldots,118$.
\begin{center}
\begin{tabular}{c|ccc|ccc}
    $t-s$ &  
    $C_{34}$ & $C_{35}$ & $C_{36}$ & $E_0C_{36}$ & $E_0C_{35}$ & $E_0C_{34}$ \\
    \hline
    114 & 3187 & 2683 & 2471 & 235 & 255 & 322 \\
    115 & 3345 & 2817 & 2587 & 251 & 273 & 342 \\
    116 & 3501 & 2946 & 2712 & 261 & 283 & 356 \\
    117 & 3666 & 3075 & 2829 & 269 & 295 & 368 \\
    118 & 3844 & 3225 & 2961 & 286 & 314 & 390
\end{tabular}
\end{center}
The corresponding Figure \ref{fig:decompa1} also shows the dimension
of the matrices in the lifting problems. The figure shows that the decomposition
of the $C\us$ into the $E_RC\us$ is not as uniform as in Figure \ref{fig:decompb}.

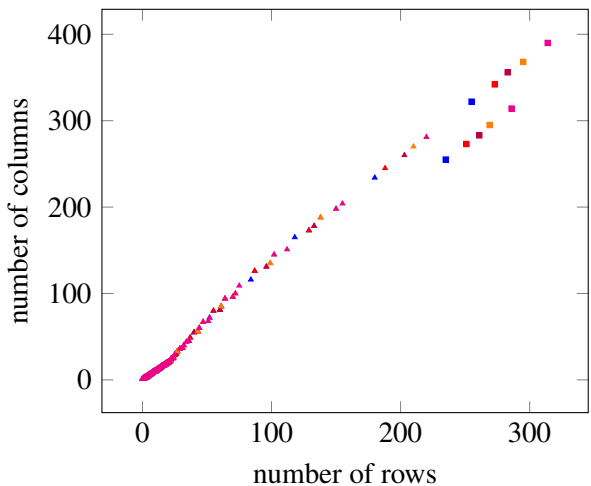
\begin{figure}
    \begin{center}
\begin{tikzpicture}
    \begin{axis}[%
        xlabel={number of rows}, ylabel={number of columns},%
        scatter/classes={%
        hom35.0.298={mark=square*,  scale=.5,blue},%
        moh35.0.298={mark=square*,  scale=.5,blue},%
        lft35.0.298={mark=triangle*,scale=.5,blue},%
        hom35.0.300={mark=square*,  scale=.5,red},%
        moh35.0.300={mark=square*,  scale=.5,red},%
        lft35.0.300={mark=triangle*,scale=.5,red},%
        hom35.0.302={mark=square*,  scale=.5,purple},%
        moh35.0.302={mark=square*,  scale=.5,purple},%
        lft35.0.302={mark=triangle*,scale=.5,purple},%
        hom35.0.304={mark=square*,  scale=.5,orange},%
        moh35.0.304={mark=square*,  scale=.5,orange},%
        lft35.0.304={mark=triangle*,scale=.5,orange},%
        hom35.0.306={mark=square*,  scale=.5,magenta},%
        moh35.0.306={mark=square*,  scale=.5,magenta},%
        lft35.0.306={mark=triangle*,scale=.5,magenta},%
            c={mark=o,draw=black}}]
            \addplot[scatter,only marks, scatter src=explicit symbolic] table[meta=label] {log.35.0.114};
            \addplot[scatter,only marks, scatter src=explicit symbolic] table[meta=label] {log.35.0.115};
            \addplot[scatter,only marks, scatter src=explicit symbolic] table[meta=label] {log.35.0.116};
            \addplot[scatter,only marks, scatter src=explicit symbolic] table[meta=label] {log.35.0.117};
            \addplot[scatter,only marks, scatter src=explicit symbolic] table[meta=label] {log.35.0.118};
        \end{axis}
\end{tikzpicture}
\end{center}
\caption{Application of Algorithm \ref{alg:new} to the computation of 
$C_s$ for $s=35$
and $t-s=114,\ldots,118$ with $B=F(2)$ (the corresponding colors are blue, red, purple, orange, magenta). 
The squares again represent 
the matrices of the initial homology computation, the triangles the lifting problems.}
\label{fig:decompa1}
\end{figure}

There is a small but interesting extension of this result for $p=2$. Let
$$F'(n) = \FF_2\cnsetof{Sq(R)}{r_1=\cdots=r_{n-1}=0,\, r_{n}\equiv 0\xmod 2} \subset F(n-1)$$
Although this is just a subalgebra of $A$, not a sub Hopf algebra, our theory is nonetheless
applicable:
\begin{thm}\label{thm:above2}
Suppose $B\subset F'(n)$ and $t<(2^{n+1}-2)s$.
Then algorithm \ref{alg:new} is applicable.
\end{thm}
\begin{proof}
See \cite{zbMATH02190992}, Satz~2.2.22.
\end{proof}
This theorem is best understood by comparison with the odd-primary situation. Recall that 
for $p>2$ the Steenrod algebra has generators $P(R)$ that resemble the $Sq(2R)$,
and separate Bockstein operators $Q(\varepsilon)$ where $\varepsilon=(\varepsilon_0,\varepsilon_1,\ldots)$
with $\varepsilon_j\in\{0,1\}$. The analogues of $F(n)$ and $F'(n)$ are
\begin{align*}
    F(n) &= \FF_p\cnsetof{Q(\varepsilon)P(R)}{r_1=\cdots=r_n=0,\,\varepsilon_0=\cdots=\varepsilon_{n-1}=0}, \\
    F'(n) &= \FF_p\cnsetof{Q(\varepsilon)P(R)}{r_1=\cdots=r_n=0,\,\varepsilon_0=\cdots=\varepsilon_{n}=0}.
\end{align*}
For $p>2$ both $F(n)$ and $F'(n)$ are sub Hopf algebras for which the analogue of Theorem \ref{thm:above} applies.

Using $F'(n)$ can give a considerable speedup. Consider again the reduction (\ref{eq:fqex})
that was applicable above the $v_1$-line of slope $1/2$.
Theorem \ref{thm:above2} allows to replace $F(1)$ by $F'(1)$
if the bidegree lies above the $h_1$-line of slope $1$.
Since $F'(1)\lefthopfquot A \cong A(0)$ this reduces the verification of
$\Ext_A^{s,t}=0$ for $t-s>1$ and $t<s$ to the empty computation!

Similarly, using $F'(2)$ is justified above a line of slope $1/5$ which eventually covers
the entire $v_1$-periodic region. To verify the exactness of the resolution in that region
the algorithm only looks at 
$$F'(2)\lefthopfquot A \cong \FF_2\cnsetof{Sq(r,\varepsilon)}{r\ge 0,\,\varepsilon=0,1}$$
which has at most $2$ generators in every degree.

\section{Lifting problems}
It is important to realize that the signature filtration does not just allow the
computation of a resolution; it also facillitates the computation {\em with} it.
The central problem here is usually to solve lifting problems: a cycle $z\in C_{s,t}$
is given and the task is to find some $w\in C_{s+1,t}$ with $d(w)=z$. 
This process is the core, for example, of the computation of a chain map when the target is $C\us$.
The signature filtration can be used to decompose this computation in exactly the same
way as during the computation of the resolution (see Algorithm \ref{alg:lift} for a formalization).

The author's experience seems to suggest that the computation of the matrices
$E_0C_{s,t}\rightarrow E_0C_{s-1,t}$ takes considerably more time than the linear algebra
routines. This suggests that it might be wise to already 
cache these matrices during the computation of the resolution.

\begin{center}
    \begin{algorithm}\label{alg:lift}
        \SetAlgoLined
        \DontPrintSemicolon
        \KwIn{A resolution up to bidegree $(s,t)$}
        \KwIn{An admissible subalgebra $B\subset A$ with its ordering of $\Sig_B$}
        \KwIn{A cycle $z\in C_{s,t}$}
        \KwResult{A preimage $w\in C_{s+1,t}$ with $d(w)=z$}
        $w \longleftarrow 0$\;
        \LineComment*[h]{Extract error terms $e$ from $z$ and compute correction to $w$}\;
        \For{$R\in\Sig_B$ (process these in order)}{
            $M \longleftarrow$ matrix of $d:E_RC_{s+1,t}\rightarrow E_RC_{s,t}$\;
            $e \longleftarrow$ extract the summands of $z$ from $E_RC_{s,t}$\;
            $f \longleftarrow$ solution of $M\cdot f = e$\;
            $w \longleftarrow w + f$\;
            $z \longleftarrow z - d(f)$\;
        }
        \caption{Computing the lift of a cycle}
    \end{algorithm}
\end{center}

\section{Implementations}
The author implemented the algorithm for $p=2$ as part of his PhD thesis. 
That implementation was written in C. 
He was then able to compute the resolution up to $t-s=210$ using a 
machine with a 300 MHz AMD K6-2 processor;
the computation took 108 days.
The running time per dimension seemed to double every 10 dimensions.
In dimension 181 the author decided to rewrite part of the multiplication routine
using the SSE2 instruction set; this essentially doubled the speed of the program.
However, given the exponential growth of the computational challenge this
was effectively only good enough to buy ten more stems.
More information about running times, 
memory usage, etc. can be found in \cite[Fig.~2.12-2.15]{zbMATH02190992}.

In 2004 the author started a new project \cite{steenrodlib} which also works
for (small) odd primes and comes equipped with an interface to the Tcl programming 
language. The tables in this paper were computed using that library. 
The library is the computational engine behind the author's Yacop/Sage project
which strives to create an experimental Steenrod algebra cohomology 
package for the Sage computer algebra system.

\section{Loose ends}
We close this paper with a few remarks.

Firstly, our algorithm seems to be a strong argument in favour of the Milnor basis
of the Steenrod algebra as the right basis for mechanical cohomology calculations. 
Our signature filtration seems to be difficult to handle in the Serre-Cartan basis
of admissible monomials, for example. 
Furthermore, Lemma \ref{lem:milmult} shows how to compute the induced differential 
on the $E_RC\us$ directly. The naive approach, i.e.~first computing the differential in $F_RC\us$
and then reducing modulo the $F_SC\us$ with $S>R$, would be a lot more wasteful.

Secondly, the algorithm can also be used to compute a 
minimal \enquote{complex motivic resolution} $M\us$
of the Steenrod algebra (see \cite{MR2629898} for an introduction to this topic).
The reason is that such an $M\us$ is just a resolution of the ordinary Steenrod 
algebra that is equipped with an extra \enquote{Bockstein filtration}.
One begins by computing a minimal resolution $K\us$ of the trigraded algebra $EA$
which is 
the \enquote{odd primary Steenrod algebra for $p=2$}.
Our theory is directly applicable to this computation.
In a second pass one then interprets $EA$ as an associated graded 
of the ordinary Steenrod algebra.
Using the same generators as in $K\us$ and lifting the terms of the 
differentials in $K\us$ arbitrarily from $EA$ to $A$ 
defines an approximate $M\us$ which however fails to satisfy $d^2(g)=0$.
One then computes successively correction terms to the $d(g)$ that remedy this.
The correction process procedes along the third (Bockstein) grading of $EA$
and requires at each step to solve a lifting problem in $K\us$
for which our theory is again applicable.

Furthermore, the signature filtration can also be used to resolve modules
other than the ground field. To find out which $B$ is applicable at a given
bidegree some external knowledge of $\Ext_B(M)$ is necessary.
This might be interesting for modules that are free over some $A(n)$, for example.
The author has not pursued this, though, since his preferred approach to 
the computation of $\Ext_A(M)$ uses the resolution $M\land C\us$ 
of $M$ where $C\us$ is the ground field resolution. At least for small $M$
that makes the computation of a separate resolution for each $M$ unnecessary.

Finally, a very interesting open question is the applicability of our shortcut to
the computation of unstable cohomology charts.

\bibhere 


\def\cprime{$'$}
\begin{thebibliography}{9}
\providecommand{\natexlab}[1]{#1}
\providecommand{\url}[1]{\texttt{#1}}
\expandafter\ifx\csname urlstyle\endcsname\relax
  \providecommand{\doi}[1]{doi: #1}\else
  \providecommand{\doi}{doi: \begingroup \urlstyle{rm}\Url}\fi

\bibitem[Adams(1966)]{MR194486}
J.~F. Adams.
\newblock A periodicity theorem in homological algebra.
\newblock \emph{Proc. Cambridge Philos. Soc.}, 62:\penalty0 365--377, 1966.
\newblock ISSN 0008-1981.
\newblock \doi{10.1017/s0305004100039955}.
\newblock URL \url{https://doi.org/10.1017/s0305004100039955}.

\bibitem[Bruner(1993)]{MR1224908}
Robert~R. Bruner.
\newblock {${\rm Ext}$} in the nineties.
\newblock In \emph{Algebraic topology ({O}axtepec, 1991)}, volume 146 of
  \emph{Contemp. Math.}, pages 71--90. Amer. Math. Soc., Providence, RI, 1993.
\newblock \doi{10.1090/conm/146/01216}.
\newblock URL \url{https://doi.org/10.1090/conm/146/01216}.

\bibitem[Dugger and Isaksen(2010)]{MR2629898}
Daniel Dugger and Daniel~C. Isaksen.
\newblock The motivic {A}dams spectral sequence.
\newblock \emph{Geom. Topol.}, 14\penalty0 (2):\penalty0 967--1014, 2010.
\newblock ISSN 1465-3060.
\newblock \doi{10.2140/gt.2010.14.967}.
\newblock URL \url{https://doi.org/10.2140/gt.2010.14.967}.

\bibitem[Margolis(1983)]{MR738973}
H.~R. Margolis.
\newblock \emph{Spectra and the {S}teenrod algebra}, volume~29 of
  \emph{North-Holland Mathematical Library}.
\newblock North-Holland Publishing Co., Amsterdam, 1983.
\newblock ISBN 0-444-86516-0.
\newblock Modules over the Steenrod algebra and the stable homotopy category.

\bibitem[Milnor(1958)]{MR0099653}
John Milnor.
\newblock The {S}teenrod algebra and its dual.
\newblock \emph{Ann. of Math. (2)}, 67:\penalty0 150--171, 1958.
\newblock ISSN 0003-486X.
\newblock URL \url{https://doi.org/10.2307/1969932}.

\bibitem[Nassau(2002)]{zbMATH02190992}
Christian Nassau.
\newblock \emph{{Ein neuer Algorithmus zur Untersuchung der Kohomologie der
  Steenrod-Algebra.}}
\newblock Berlin: Logos Verlag; Frankfurt am Main: Fachbereich Mathematik
  (Dissertation 2001), 2002.
\newblock ISBN 3-89722-881-5/pbk.
\newblock \doi{10.30819/881}.
\newblock URL \url{https://doi.org/10.30819/881}.

\bibitem[Nassau(2019)]{steenrodlib}
Christian Nassau.
\newblock The {S}teenrod {T}cl library v2.1, October 2019.
\newblock URL \url{https://doi.org/10.5281/zenodo.3473101}.

\bibitem[Ravenel(1986)]{MR860042}
Douglas~C. Ravenel.
\newblock \emph{Complex cobordism and stable homotopy groups of spheres},
  volume 121 of \emph{Pure and Applied Mathematics}.
\newblock Academic Press, Inc., Orlando, FL, 1986.
\newblock ISBN 0-12-583430-6; 0-12-583431-4.

\bibitem[Tangora(1985)]{MR818916}
Martin~C. Tangora.
\newblock Computing the homology of the lambda algebra.
\newblock \emph{Mem. Amer. Math. Soc.}, 58\penalty0 (337):\penalty0 v+163,
  1985.
\newblock ISSN 0065-9266.
\newblock \doi{10.1090/memo/0337}.
\newblock URL \url{https://doi.org/10.1090/memo/0337}.

\end{thebibliography}
\end{document}